\documentclass[twoside, 12pt]{article}
\usepackage{amssymb, amsmath, mathrsfs, amsthm}
\usepackage{graphicx}
\usepackage{color}
\usepackage{float}

\newcommand{\bd}{\begin{description}}
\newcommand{\ed}{\end{description}}
\newcommand{\bi}{\begin{itemize}}
\newcommand{\ei}{\end{itemize}}
\newcommand{\bt}{\begin{tabular}}
\newcommand{\et}{\end{tabular}}
\newcommand{\beq}{\begin{equation}}
\newcommand{\eeq}{\end{equation}}
\newcommand{\be}{\begin{enumerate}}
\newcommand{\ee}{\end{enumerate}}
\newcommand{\beqs}{\begin{eqnarray*}}
\newcommand{\eeqs}{\end{eqnarray*}}

\definecolor{DarkGreen}{rgb}{0.2, 0.6, 0.3}

\newtheorem{theorem}{Theorem}
\newtheorem{lemma}{Lemma}

\newtheorem{case}{Case}

\newtheorem{claim}{Claim}

\begin{document}
\title{\bf Gallai-Ramsey numbers of odd cycles \footnote{In July, 2018, Dr. Zixia Song claimed in a conference in Xining, China that she and her coauthors had obtained results for Gallai-Ramsey numbers of odd cycles. Their work was independent of the present work and their submission soon followed the initial submission of this manuscript.}\footnote{Supported by the National Science Foundation of China (Nos. 12061059, 11601254, 11551001, 11161037, 61763041, 11661068, and 11461054) and the Qinghai Key Laboratory of Internet of Things Project (2017-ZJ-Y21).}}

\author{
Zhao Wang\footnote{College of Science, China Jiliang University, Hangzhou 310018, China. {\tt wangzhao@mail.bnu.edu.cn}},
Yaping Mao\footnote{School of Mathematics and Statistics, Qinghai Normal University, Xining, Qinghai 810008, China. {\tt maoyaping@ymail.com}} \footnote{Academy of Plateau Science and Sustainability, Xining, Qinghai 810008, China},
Colton Magnant\footnote{Advanced Analytics Group, United Parcel Service, Atlanta, GA, 30328, USA. {\tt dr.colton.magnant@gmail.com}} \footnotemark[4],
\\
Ingo Schiermeyer\footnote{Technische Universit{\"a}t Bergakademie Freiberg, Institut f{\"u}r Diskrete Mathematik und Algebra, 09596 Freiberg, Germany. {\tt Ingo.Schiermeyer@tu-freiberg.de}},
Jinyu Zou\footnote{School of Computer, Qinghai Normal University, Xining, Qinghai 810008, China. {\tt zjydjy2015@126.com}}.
}

\maketitle

\begin{abstract}
Given two graphs $G$ and $H$ and a positive integer $k$, the
$k$-color Gallai-Ramsey number, denoted by ${\rm gr}_{k}(G : H)$, is
the minimum integer $N$ such that for all $n \geq N$, every
$k$-coloring of the edges of $K_{n}$ contains either a rainbow copy
of $G$ or a monochromatic copy of $H$. We prove that ${\rm gr}_{k}
(K_{3} : C_{2\ell + 1}) = \ell \cdot 2^{k} + 1$ for all $k \geq 1$
and $\ell \geq 6$.
\end{abstract}

\section{Introduction}\label{Sec:Intro}

All graphs considered in this work are simple, with no loops or
multiple edges. By a \emph{coloring} of a graph, we mean a coloring
of the edges of the graph. For a graph $G$, let $e(G)$ be the size
of $G$. A colored graph is called \emph{rainbow} if every edge
receives a distinct color. All standard notation used here comes
from \cite{CLZ11}.

Given two graphs $G$ and $H$ and a positive integer $k$, the
$k$-color Gallai-Ramsey number, denoted by ${\rm gr}_{k}(G : H)$, is
the minimum integer $N$ such that for all $n \geq N$, every
$k$-coloring of $K_{n}$ contains either a rainbow copy of $G$ or a
monochromatic copy of $H$.

Given two graphs $G$ and $H$, let $R(G, H)$ denote the $2$-color
Ramsey number for finding a monochromatic $G$ or $H$, that is, the
minimum number of vertices $n$ needed so that every red-blue
coloring of $K_{n}$ contains either a red copy of $G$ or a blue copy
of $H$.

The theory of Gallai-Ramsey numbers has grown by leaps and bounds in
recent years, especially for the case where $G = K_{3}$. Many
individual graphs have been solved for any number of colors but
classes of graphs have generally only been bounded from above and
below. We refer the interested reader to the survey \cite{MR2606615}
(with an updated version at \cite{FMO14}) for more general
information. In particular, the following has been shown for odd
cycles.

\begin{theorem}[\cite{FM11,MR3121112}]\label{Thm:OddBounds}
Given integers $n \geq 2$ and $k \geq 1$,
$$
n2^{k} + 1 \leq {\rm gr}_k(K_3:C_{2n + 1}) \leq (2^{k + 3} - 3)n
\log n.
$$
\end{theorem}

For the triangle $C_{3} = K_{3}$, several authors obtained the
following result.

\begin{theorem}[\cite{AI08,MR729784,GSSS10}]
For $k \geq 2$,
$$
{\rm gr}_{k}(K_{3} : K_{3}) = \begin{cases}
5^{k/2} + 1 & \text{ if $k$ is even,}\\
2 \cdot 5^{(k-1)/2} + 1 & \text{ if $k$ is odd}.\\
\end{cases}
$$
\end{theorem}

For the $C_{5}$, Fujita and Magnant obtained the following.

\begin{theorem}[\cite{FM11}]
For any positive integer $k \geq 2$, we have
$$
{\rm gr}_k(K_3:C_5)=2^{k+1}+1.
$$
\end{theorem}

More recently, other authors have extended these results to longer
odd cycles in a sequence of papers.

\begin{theorem}[\cite{Song2, Song1}]
For integers $3 \leq \ell \leq 5$ and $k \geq 1$, we have
$$
{\rm gr}_{k} (K_{3} : C_{2\ell + 1}) = \ell \cdot 2^{k} + 1.
$$
\end{theorem}

Based on these results, it has been widely believed that the lower
bound of Theorem~\ref{Thm:OddBounds} is indeed the correct value. In
this paper, we confirm that belief with the following, our main
result.

\begin{theorem}\label{Thm:OddCycles}
For integers $\ell \geq 6$ and $k \geq 1$, we have
$$
{\rm gr}_{k} (K_{3} : C_{2\ell + 1}) = \ell \cdot 2^{k} + 1.
$$
\end{theorem}

\section{Preliminaries}\label{Sec:Prelim}

We begin this section with the fundamental tool in the study of
colored complete graphs containing no rainbow triangle. Here a
colored complete graph is called \emph{Gallai colored} if it
contains no rainbow triangle.

\begin{theorem}[\cite{MR1464337,MR0221974,MR2063371}]\label{Thm:G-Part}
In any Gallai colored complete graph, there exists a partition of
the vertices (called a \emph{Gallai partition}) such that there are
at most two colors on the edges between the parts and only one color
on edges between each pair of parts.
\end{theorem}

We will use this result as a matter of fact. Since a subgraph
consisting of an arbitrarily selected vertex from each part of a
Gallai partition is a $2$-colored complete graph, we call this
subgraph the \emph{reduced graph} of the Gallai partition.

\begin{lemma}\label{Lemma:SmallOutside}
Let $k \geq 1$ be an integer, $H$ be a graph with $|H| = m$, and let
$G$ be a complete graph $K_{n}$ with a Gallai coloring containing no
monochromatic copy of $H$. If the vertex set is partitioned into
$V(G) = A \cup B_{1} \cup B_{2} \cup \dots \cup B_{k - 1}$ where the
set $A$ uses at most $k$ colors (say from the set $[k]$ of integers
from $1$ to $k$), $|B_{i}| \leq m - 1$ holds for all $i$, and all
edges between $A$ and $B_{i}$ have color $i$, then $n \leq {\rm
gr}_{k}(K_{3} : H) - 1$.
\end{lemma}

Note that this lemma uses the assumed structure to provide a bound
on $|G|$ even if $G$ itself uses more than $k$ colors.

\begin{proof}
If $k=1$, then $V(G)=A$, and hence $n \leq {\rm gr}_{1}(K_{3} : H) -
1$. If $k=2$, then $|B_{1}| \leq m - 1 < |H|$, and hence $n \leq
{\rm gr}_{2}(K_{3} : H) - 1$. We assume that $k \geq 3$. For $i \neq
j$, all edges between $B_{i}$ and $B_{j}$ must have either color $i$
or color $j$ to avoid a rainbow triangle. Since $|B_{i}| \leq m - 1
< |H|$, changing all edges within $B_{i}$ that are not colored with
a color in $[k - 1]$ to color $k$ cannot create a monochromatic copy
of $H$. The result of this modification is a $k$-colored copy of
$K_{n}$ with no rainbow triangle and no monochromatic copy of $H$ so
$n \leq {\rm gr}_{k}(K_{3} : H) - 1$.
\end{proof}

We use this classical result of Dirac to obtain an easy lemma.

\begin{theorem}[\cite{MR0047308}]\label{Thm:Dirac}
Let $G$ be a graph of order $n \geq 3$. If the minimum degree of $G$ satisfies $\delta(G) \geq \frac{n}{2}$, then $G$ is hamiltonian.
\end{theorem}

\begin{lemma}\label{Lemma:2l+1}
If there are at least $2\ell + 1$ vertices in a Gallai colored
complete graph $G$ with only one color appearing on edges between
parts of the Gallai partition and all parts of order at most $\ell$,
then $G$ contains a monochromatic copy of $C_{2\ell + 1}$.
\end{lemma}

\begin{proof}
Let blue be the color of the edges between parts. Since each part of
the Gallai partition has order at most $\ell$, every vertex is
incident with at least $|G|-\ell\geq |G|/2$ blue edges because
$|G|\geq 2\ell+1$. By Theorem~\ref{Thm:Dirac}, $G$ contains a blue
copy of $C_{2\ell + 1}$.
\end{proof}

Next is a helpful result on the existence of paths.

\begin{theorem}[\cite{MR0114772}]\label{Thm:E-G}
Let $G$ be a graph on $n$ vertices and let $k \geq 2$ be an integer.
If the number of edges in $G$ satisfies $e(G) > \frac{k - 1}{2} n$,
then $G$ contains a path with $k$ edges (i.e.~a copy of $P_{k +
1}$).
\end{theorem}

We use Theorem~\ref{Thm:E-G} to prove the following colored version.
Let $d_{R}(v)$ and $d_{B}(v)$ denote the red and blue degrees of the
vertex $v$ respectively, that is, the number of edges incident to
$v$ that have color red and blue respectively.

\begin{lemma}\label{Lemma:ColE-G}
Let $a,b$ be two integers with $a\geq 3$ and $b\geq 3$. Let $G$ be a
graph of order $n$ with edges colored by red and blue. If for every
vertex $v \in V(G)$ and for some non-negative integers $a$ and $b$
with $a + b \geq 3$, we have $d(v) \geq a + b - 3$, then $G$
contains either a red copy of $P_{a}$ or a blue copy of $P_{b}$.
\end{lemma}

\begin{proof}
Let $\bar{d}(G)$, $\bar{d}_{R}(G)$, and $\bar{d}_{B}(G)$ denote the
average degree, average red degree, and average blue degree of $G$
respectively. Then $\bar{d}(G) = \bar{d}_{R}(G) + \bar{d}_{B}(G)$.
Since $d_{R}(v) + d_{B}(v) \geq a + b - 3$ for every vertex $v \in
V(G)$, we have $\bar{d}_{R}(G) + \bar{d}_{B}(G) \geq a + b - 3$ so
either $\bar{d}_{R}(G) > a - 2$ or $\bar{d}_{B}(G) > b - 2$. If
$\bar{d}_{R}(G) > a - 2$, then there are more than $\frac{n(a -
2)}{2}$ red edges in $G$, so by Theorem~\ref{Thm:E-G}, $G$ contains
a red copy of $P_{a}$. On the other hand, if $\bar{d}_{B}(G) > b -
2$, then there are more than $\frac{n(b - 2)}{2}$ blue edges in $G$,
so by Theorem~\ref{Thm:E-G}, $G$ contains a blue copy of $P_{b}$,
completing the proof.
\end{proof}

\begin{lemma}\label{Lemma:5l-2}
Let $b,\ell$ be two integers. Suppose there are at least $4\ell+b-2$
vertices in a Gallai colored complete graph $G$ with only two
colors, say red and blue, appearing on edges between parts of the
Gallai partition, and all parts of order at most $b$. If there
exists a part $H_1$ such that $|H_1|\geq 2$ and all the edges from
$H_1$ to the other parts are red or blue, then $G$ contains a
monochromatic copy of $C_{2\ell + 1}$.
\end{lemma}

\begin{proof}
Let $H_{1}, H_{2}, \ldots, H_{t}$ be parts of this Gallai partition.
Note that the edges between parts of the Gallai partition are red or
blue. Without loss of generality, we assume that all the edges from
$H_1$ to $H_{2}\cup \cdots \cup H_{t}$ are red. Let $|H_1|=x$. Then
$x\geq 2$. Let $P=u_1u_2\cdots u_{y}$ be a longest red path of
$H_{2}\cup \cdots \cup H_{t}$. If $y\geq 2\ell$, then we choose a
vertex from $H_1$ and this vertex together with $P$ form a cycle
$C_{2\ell + 1}$, as desired. Assume $y\leq 2\ell-1$. Clearly,
$\sum_{i=2}^{t}|H_i|-y\geq 4\ell+b-2-x-y=2\ell+1-x-y+2\ell-3+b$.
Since $|H_i|\leq b$ for each $i \ (1\leq i\leq t)$, it follows that
for each vertex $v\in \bigcup_{i=2}^{t}H_i-V(P)$, $d_R(v)+d_B(v)\geq
2\ell+1-x-y+2\ell-3$. From Lemma \ref{Lemma:ColE-G},
$\bigcup_{i=2}^{t}H_i-V(P)$ contains a red path of length
$2\ell-x-y$ or a blue path of length $2\ell-1$.

Suppose that there is a red path of length $2\ell-x-y$ in
$\bigcup_{i=2}^{t}H_i-V(P)$, say $P_{2\ell+1-x-y}=v_{1}v_{2}\ldots
v_{2\ell+1-x-y}$. Let $V(H_1)=\{w_1,w_2,\ldots w_{x}\}$. Then
$w_1u_1u_2\ldots u_yw_2v_1v_2\ldots
v_{2\ell-2x-y+3}w_3v_{2\ell-2x-y+4}w_4v_{2\ell-2x-y+5}w_5\ldots
v_{2\ell-x-y}w_x$\\ $v_{2\ell-x-y+1}w_1$ is a red $C_{2\ell + 1}$,
as desired.

Suppose that there is a blue path of length $2\ell-1$ in
$\bigcup_{i=2}^{t}H_i-V(P)$, say $P_{2\ell}=v_{1}v_{2}\ldots
v_{2\ell}$. Without loss of generality, let $u_1\in V(H_j) \ (1\leq
j\leq t)$. If $v_{1}\not\in V(H_j)$ and $v_{2\ell}\not\in V(H_j)$,
then $u_1v_1$ and $u_1v_{2\ell}$ are blue, and hence
$u_1v_1v_2\ldots v_{2\ell}u_1$ is a blue $C_{2\ell + 1}$, as
desired. If $v_{1}\in V(H_j)$ and $v_{2\ell}\in V(H_j)$, then we
choose a vertex $w\in \bigcup_{i=2}^{t}H_i-V(P)-V(P_{2\ell})$, and
the edge $wv_1$ and $wv_{2\ell}$ are blue, and hence $wv_1v_2\ldots
v_{2\ell}w$ is a blue $C_{2\ell + 1}$, as desired. Suppose that
$v_{1}\in V(H_j)$ and $v_{2\ell}\not\in V(H_j)$. Then $v_1v_{2\ell}$
is blue, and hence $v_1v_2\ldots v_{2\ell}v_1$ is a blue
$C_{2\ell}$. Then there exists two vertices $v_{a},v_{a+1}$ in the
cycle $C_{2\ell}$ such that $v_{a}\not\in V(H_j)$ and
$v_{a+1}\not\in V(H_j)$. Since the edges $u_1v_a$ and $u_1v_{a+1}$
are blue, it follows that $C_{2\ell}-v_{a}v_{a+1}+u_1v_a+v_{a+1}$ is
a blue $C_{2\ell + 1}$, as desired. The same is true for the case
that $v_{1}\not\in V(H_j)$ and $v_{2\ell}\in V(H_j)$.
\end{proof}

\begin{lemma}\label{Lemma:7l-3}
Suppose there are at least $7\ell-3$ vertices in a Gallai colored
complete graph $G$ with only two colors, say red and blue, appearing
on edges between parts of the Gallai partition, and all parts of
order at most $\ell$. If there exists a part $H_1$ such that
$|H_1|=1$ and all the edges from $H_1$ to the other parts are red or
blue, then $G$ contains a monochromatic copy of $C_{2\ell + 1}$.
\end{lemma}
\begin{proof}
Let $H_{1}, H_{2}, \ldots, H_{t}$ be parts of this Gallai partition.
Note that the edges between parts of the Gallai partition are red or
blue. Without loss of generality, we assume that all the edges from
$H_1$ to $H_{2}\cup \cdots \cup H_{t}$ are red. Let $V(H_1)=\{w_1\}$. Let $P=u_1u_2\cdots u_{y}$ be a longest red path of
$H_{2}\cup \cdots \cup H_{t}$. If $y\geq 2\ell$, then we choose a
vertex from $H_1$ and this vertex together with $P$ form a cycle
$C_{2\ell + 1}$, as desired. Assume $y\leq 2\ell-1$. Clearly,
$\sum_{i=2}^{t}|H_i|-y\geq 7\ell-3-1-y\geq 7\ell-3-1-2\ell+1= 2\ell+2\ell-3+\ell$.
Since $|H_i|\leq \ell$ for each $i \ (1\leq i\leq t)$, it follows
that for each vertex $v\in \bigcup_{i=2}^{t}H_i-V(P)$,
$d_R(v)+d_B(v)\geq 2\ell+2\ell-3$. From Lemma
\ref{Lemma:ColE-G}, $\bigcup_{i=2}^{t}H_i-V(P)$ contains a red path
of length $2\ell-1$ or a blue path of length $2\ell-1$.

Suppose that there is a red path of length $2\ell-1$ in
$\bigcup_{i=2}^{t}H_i-V(P)$, say $P_{2\ell}=v_{1}v_{2}\ldots
v_{2\ell}$. Then
$w_1v_{1}v_{2}\ldots
v_{2\ell}w_1$ is a red $C_{2\ell + 1}$,
as desired.

Suppose that there is a blue path of length $2\ell-1$ in
$\bigcup_{i=2}^{t}H_i-V(P)$, say $P_{2\ell}=v_{1}v_{2}\ldots
v_{2\ell}$. Without loss of generality, let $u_1\in V(H_j) \ (1\leq
j\leq t)$. If $v_{1}\not\in V(H_j)$ and $v_{2\ell}\not\in V(H_j)$,
then $u_1v_1$ and $u_1v_{2\ell}$ are blue, and hence
$u_1v_1v_2\ldots v_{2\ell}u_1$ is a blue $C_{2\ell + 1}$, as
desired. If $v_{1}\in V(H_j)$ and $v_{2\ell}\in V(H_j)$, then we
choose a vertex $w\in \bigcup_{i=2}^{t}H_i-V(P)-V(P_{2\ell})$, and
the edge $wv_1$ and $wv_{2\ell}$ are blue, and hence $wv_1v_2\ldots
v_{2\ell}w$ is a blue $C_{2\ell + 1}$, as desired. Suppose that
$v_{1}\in V(H_j)$ and $v_{2\ell}\not\in V(H_j)$. Then $v_1v_{2\ell}$
is blue, and hence $v_1v_2\ldots v_{2\ell}v_1$ is a blue
$C_{2\ell}$. Then there exists two vertices $v_{a},v_{a+1}$ in the
cycle $C_{2\ell}$ such that $v_{a}\not\in V(H_j)$ and
$v_{a+1}\not\in V(H_j)$. Since the edges $u_1v_a$ and $u_1v_{a+1}$
are blue, it follows that $C_{2\ell}-v_{a}v_{a+1}+u_1v_a+v_{a+1}$ is
a blue $C_{2\ell + 1}$, as desired. The same is true for the case
that $v_{1}\not\in V(H_j)$ and $v_{2\ell}\in V(H_j)$.
\end{proof}

Finally, we state some other known results that will be used in the proof.

\begin{theorem}[\cite{MR0345866, MR1846919, MR0332567}]\label{Thm:2-color}
$$
R(C_{m}, C_{n}) = \begin{cases}
2n - 1 \\
~ ~ ~ ~ \text{ for $3 \leq m \leq n$, $m$ and $n$ odd, $(m, n) \neq (3, 3)$,}\\
n - 1 + m/2 \\
~ ~ ~ ~ \text{ for $4 \leq m \leq n$, $m$ and $n$ even, $(m, n) \neq (4, 4)$,}\\
\max\{ n - 1 + m/2, 2m - 1\}\\
~ ~ ~ ~ \text{ for $4 \leq m < n$, $m$ even and $n$ odd.}
\end{cases}
$$
\end{theorem}

\begin{theorem}[\cite{FM11, MR3121112}]\label{Thm:EvenUp}
Given integers $n \geq 2$ and $k \geq 1$,
$$
(n - 1)k + n + 1 \leq {\rm gr}_{k} (K_{3} : C_{2n}) \leq (n - 1)k +
3n.
$$
\end{theorem}

In particular, we mostly use the immediate corollary of
Theorem~\ref{Thm:EvenUp} that ${\rm gr}_{3}(K_{3} : C_{2\ell - 2})
\leq 6\ell - 9$.

\section{Proof of Theorem~\ref{Thm:OddCycles}} \label{Sec:OddCyclesPf}

\begin{proof}
This proof is by induction on $k$. Throughout the proof, we will eliminate edges of some colors from subgraphs and then apply induction on the number of colors available for use on the edges within those subgraphs to provide bounds on the total number of vertices. Since the cases $k = 1$ and $k = 2$ are either trivial or follow from Theorem~\ref{Thm:2-color} respectively, we assume $k \geq 3$. For a contradiction, suppose $\ell \geq 3$ and let $G$ be a Gallai colored complete graph $K_{n}$ containing no monochromatic copy of $C_{2\ell + 1}$ with
$$
n = \ell \cdot 2^{k} + 1.
$$
The goal is to arrive at a contradiction.

Let $T$ be a maximal set of vertices $T = T_1 \cup T_2 \ldots T_k$ where each subset $T_I$ has all edges
to $G \setminus T$ in color $i$ and $|G \setminus T| \geq \ell$ constructed iteratively by adding at most $2\ell$ vertices at a time, with at most $\ell$ vertices being added to two sets $T_i$ at a time. In Claim \ref{Claim:Tissmall} we will show that $|T_i|$ is small for all $i.$ First, however, we give a full description how such a set $T$ can be generated iteratively by using Gallai partitions.

By Theorem~\ref{Thm:G-Part}, there exists a partition of $V(G)$ into
parts such that between each pair of parts there is exactly one
color and between the parts in general, there are at most two
colors. Consider such a partition with the smallest number of
parts, say $t_1$. Since ${\rm R}(C_{2\ell + 1}, C_{2\ell + 1})=
4\ell + 1$, it follows that $t_1\leq 4\ell$. Let $H_{1}^{1},
H_{2}^{1}, \ldots, H_{t_1}^{1}$ be parts of the partition.

If $2\leq t_1 \leq 3$, then by the minimality of $t_1$, we may
assume $t_1=2$. Suppose all edges between $H_{1}^{1}$ and
$H_{2}^{1}$ are color $c_1$. Without loss of generality, we assume
$|H_{1}^{1}|\leq |H_{2}^{1}|$. If $t_1\geq 4$, then we suppose all
edges among $H_{1}^{1}, H_{2}^{1}, \ldots, H_{t_1}^{1}$ are color
$c_{11}$ and $c_{12}$.

\begin{itemize}
\item[] $(a)$ If $t_1=2$, then we suppose that $|H_{1}^{1}|\leq \ell$.

\item[] $(b)$ If $t_1\geq 4$, then for each $i \
(1\leq i\leq t_1)$, let $A_{i}^{1}$ be the union of parts with color
$c_{11}$ to $H_{i}^{1}$ and $B_{i}^{1}$ be the union of parts with
color $c_{12}$ to $H_{i}^{1}$, and then we suppose that there exists
$H_{x_1}^{1}$ such that $|A_{x_1}^{1}|\leq \ell$ and
$|B_{x_1}^{1}|\leq \ell$.
\end{itemize}

Because there is no rainbow triangle in $H_{2}^{1}$ or
$H_{x_1}^{1}$, by Theorem~\ref{Thm:G-Part}, there exists a partition
of $V(H_{2}^{1})$ or $V(H_{x_1}^{1})$ into parts such that between
each pair of parts there is exactly one color and between the parts
in general, there are at most two colors. Consider such a
$H_{2}^{1}$-partition or $H_{x_1}^{1}$-partition with the smallest
number of parts, say $t_2$, clearly, $2\leq t_2 \leq 4\ell$. Let
$H_{1}^{2}, H_{2}^{2}, \ldots, H_{t_2}^{2}$ be parts of the
$H_{2}^{1}$-partition.

Suppose $2\leq t_2\leq 3$. By the minimality of $t_2$, we may assume
$t_2=2$. Suppose all edges between $H_{1}^{2}$ and $H_{2}^{2}$ are
color $c_2$. Without loss of generality, we assume $|H_{1}^{2}|\leq
|H_{2}^{2}|$. If $t_2\geq 4$, then we suppose all edges among
$H_{1}^{2}, H_{2}^{2}, \ldots, H_{t_2}^{2}$ are color $c_{21}$ and
$c_{22}$.

\begin{itemize}
\item[] $(a)$ If $t_2=2$, then we suppose that $|H_{1}^{2}|\leq \ell$.

\item[] $(b)$ If $t_2\geq 4$, then for each $i \ (1\leq i\leq t_2)$, let
$A_{i}^{2}$ be the set of parts with color $c_{21}$ to $H_{i}^{2}$
and $B_{i}^{2}$ be the set of parts with color $c_{22}$ to
$H_{i}^{2}$, and then we suppose that there exists $H_{x_2}^{2}$
such that $|A_{x_2}^{2}|\leq \ell$ and $|B_{x_2}^{2}|\leq \ell$.
\end{itemize}

Because there is no rainbow triangle in $H_{2}^{2}$ or
$H_{x_2}^{2}$, by Theorem~\ref{Thm:G-Part}, there exists a partition
of $V(H_{2}^{2})$ or $V(H_{x_2}^{2})$ into parts such that between
each pair of parts there is exactly one color and between the parts
in general, there are at most two colors. Consider such a
$H_{2}^{2}$-partition or $H_{x_2}^{2}$-partition with the smallest
number of parts, say $t_3$, clearly, $2\leq t_3 \leq 4\ell$. Let
$H_{1}^{3}, H_{2}^{3}, \ldots, H_{t_3}^{3}$ be parts of the
$H_{2}^{2}$-partition or $H_{x_2}^{2}$-partition.

Continue the process outlined above. Then there exists an integer $s$ such that
\begin{itemize}
\item $s$ is maximum.

\item For each $i \ (1\leq i\leq s)$, if $t_i=2$, then we suppose all edges between $H_{1}^{i}$ and $H_{2}^{i}$ are
color $c_i$; if $t_i\geq 4$, then we suppose all edges among
$H_{1}^{i}, H_{2}^{i}, \ldots, H_{t_{i}}^{i}$ are color $c_{i1}$ and
$c_{i2}$.

\item Let $H_{1}^{i}, H_{2}^{i}, \ldots, H_{t_i}^{i}$ be parts of the $H_{2}^{i-1}$-partition or $H_{x_{i-1}}^{i-1}$-partition for each $i \ (1\leq i \leq s)$;

\item For each $i \ (1\leq i\leq s)$, $|H_{1}^{i}|\leq \ell$ if $t_i=2$; there exists $H_{x_i}^{i}$ such that
$|A_{x_i}^{i}|\leq \ell$ and $|B_{x_i}^{i}|\leq \ell$ if $t_i\geq
4$.

\end{itemize}
\small\small
\begin{center}
\begin{tabular}{|c|c|c|c|}
\hline Step & Parts & Gallai-partition & Conditions \\[0.1cm]
\cline{1-4} $1$ & $G=K_n$ & $H_1^1,H_2^1,\ldots,H_{t_1}^1$ &
$\left\{
\begin{array}{ll}
|H^1_1|\leq \ell, &\mbox{{\rm if}~$t_1=2$,}\\
\exists~part~H^{1}_{x_1},s.t.,|A^{1}_{x_1}|\leq \ell,
~|B^{1}_{x_1}|\leq \ell,&\mbox{{\rm if}~$t_1\geq 4$}.
\end{array}
\right.$\\[0.1cm]
\cline{1-4} $2$ &$H^1_2$ or $H^{1}_{x_1}$ &
$H_1^2,H_2^2,\ldots,H_{t_2}^2$ & $\left\{
\begin{array}{ll}
|H^2_1|\leq \ell, &\mbox{{\rm if}~$t_2=2$,}\\
\exists~part~H^{2}_{x_2},s.t.,|A^{2}_{x_2}|\leq \ell,
~|B^{2}_{x_2}|\leq \ell,&\mbox{{\rm if}~$t_2\geq 4$}.
\end{array}
\right.$\\[0.1cm]
\cline{1-4} $3$ &$H^2_2$ or $H^{2}_{x_2}$ &
$H_1^3,H_2^3,\ldots,H_{t_3}^3$ & $\left\{
\begin{array}{ll}
|H^3_1|\leq \ell, &\mbox{{\rm if}~$t_3=2$,}\\
\exists~part~H^{3}_{x_3},s.t.,|A^{3}_{x_3}|\leq \ell,~|B^{3}_{x_3}|\leq
\ell,&\mbox{{\rm if}~$t_3\geq 4$}.
\end{array}
\right.$\\[0.1cm]
\cline{1-4} $\cdots$ & $\cdots \cdots$ & $\cdots \cdots$ &
$\cdots \cdots$\\[0.1cm]
\cline{1-4} $s$ &$H^{s-1}_2$ or $H^{s-1}_{x_{s-1}}$ &
$H_1^{s},H_2^{s},\ldots,H_{t_{s}}^{s}$ & $\left\{
\begin{array}{ll}
|H^{s}_1|\leq \ell, &\mbox{{\rm if}~$t_s=2$,}\\
\exists~part~H^{s}_{x_{s}},s.t.,|A^{s}_{x_s}|\leq \ell,
~|B^{s}_{x_s}|\leq \ell,&\mbox{{\rm if}~$t_s\geq 4$}.
\end{array}
\right.$\\[0.1cm]
\cline{1-4}
\end{tabular}
\end{center}
\begin{center}
{Table 1. The process for Theorem~\ref{Thm:OddCycles}}.
\end{center}

After the above $s$ steps, there exists one special part $H_2^s$ if $t_s=2$, or $H_{x_s}^s$ if $t_s\geq 4$. Consider all the parts outside $H_2^s$
(if $t_s=2$) or $H_{x_s}^s$ (if $t_s\geq 4$). Note that their sizes are each at most $\ell$. Let $T_i^s \ (1\leq i\leq k)$ be the union of parts with edges colored $i$ to $H_2^s$ (if $t_s=2$) or $H_{x_s}^s$ (if $t_s\geq 4$). Let $T^s  = T_{1}^s  \cup T_{2}^s  \cup \dots \cup T_{k}^s $.

\begin{claim}\label{Claim:Tissmall}
For each $i$ with $1 \leq i \leq k$, we have $|T_{i}^s | \leq \ell$
and furthermore, $T_{i}^s = \emptyset$ for some $i$.
\end{claim}

\begin{proof}
Assume, to the contrary, that there exists a $T_{i}^s $ such that
$|T_{i}^s |\geq \ell+1$. Furthermore, there exists some step $a \
(1\leq a\leq s)$ satisfying the following conditions.

\begin{itemize}
\item $a$ is minimum.

\item After the $a$ steps, there exists one part $H_2^{a}$ if $t_{a}=2$;
$H_{x_{a}}^{a}$ if $t_{a}\geq 4$. Consider all the parts outside
$H_2^{a}$ (if $t_{a}=2$) or $H_{x_{a}}^{a}$ (if $t_{a}\geq 4$). Note
that their size are at most $\ell$. Let $T_i^a \ (1\leq i\leq k)$ be
the union of parts with edges colored $i$ to $H_2^a$ (if $t_a=2$) or
$H_{x_a}^a$ (if $t_a\geq 4$). Then there exists a $T_{j}^a$ such
that $|T_{j}^a |\geq \ell+1$.
\end{itemize}

Let $T^a = T_{1}^a  \cup T_{2}^a  \cup \dots \cup T_{k}^a $. Then
$|T^a|\leq (k+2)\ell$, and $|G|-|T^a|\geq n-(k+2)\ell=\ell \cdot
2^{k}+1-(k+2)\ell>\ell$. Clearly, there exists two parts in
$T_{j}^a$, say $H',H''$, such that the edges from $H'$ to $H''$ are
colored $j$. Note that the edges from $T_{j}^a$ to $G-T^a$ are also
colored $j$. Then there is an odd cycle of length $2\ell+1$, a
contradiction.

We next show that $T_{i}^s = \emptyset$ for some $i$. If $k \geq 4$,
then $2^{k}\geq 2k+3$, and hence
$$
|G\setminus T^s| \geq [\ell 2^{k} + 1] - k\ell \geq (\ell - 1)k +
3\ell.
$$
By Theorem~\ref{Thm:EvenUp}, there is a monochromatic copy of
$C_{2\ell}$ contained within $G \setminus T^s$. Since $T_{i}^s \neq
\emptyset$ for all $i$, this cycle can easily be extended to a
monochromatic copy of $C_{2\ell + 1}$, for a contradiction. We may
therefore assume $k = 3$. Then $|G
\setminus T^s| \geq 5\ell + 1$.

Within $G' = G \setminus T^s$, consider a Gallai partition and let
$r$ be the number of parts in this partition of order at least
$\ell$ and suppose red (color $1$) and blue (color $2$) are the
colors appearing on edges between the parts of the partition. If $r
\geq 2$, then two of these large parts, say with red edges between,
along with any vertex of $T_{1}^s$, produces a red copy of $C_{2\ell
+ 1}$. This means we may immediately assume that $r \leq 1$. If $r =
1$ say with $H_{1}$ as the large part, then let $G_{R}$ and $G_{B}$
be the sets of vertices in $G' \setminus H_{1}$ with all red and
blue (respectively) edges to $H_{1}$. If $|G_{R} \cup T_{1}^s| \geq
\ell + 1$ or $|G_{B} \cup T_{2}^s| \geq \ell + 1$, then there is a
red or blue copy of $C_{2\ell + 1}$ respectively, meaning that
$|G_{R} \cup T_{1}^s|, |G_{B} \cup T_{2}^s| \leq \ell$. Then $G_{R}$
and $G_{B}$ can be added to $T^s$ to produce a larger set than $T^s$
with the same properties, contradicting the maximality of $s$. We
may therefore assume $r = 0$. Within $G'$, every vertex has degree
at least $|G'| - (\ell - 1)$ when restricted to the red and blue
edges. Since $|G'| \geq 5\ell + 1$, we have
$$
d(v) \geq 5\ell - (\ell - 1) = 4\ell + 1 > [2\ell] + [2\ell] - 3
$$
so, by Lemma~\ref{Lemma:ColE-G}, there must exist either a red path
or a blue path of order at least $2\ell$. This path along with a
vertex of $T^s$ with appropriately colored edges produces a
monochromatic copy of $C_{2\ell + 1}$, a contradiction, completing
the proof of Claim~\ref{Claim:Tissmall}.
\end{proof}

Consider a Gallai partition of $G \setminus T^s$ with the minimum
number of parts $t$ and let $H_{1}, \dots, H_{t}$ be the parts of
the partition where $|H_{1}| \geq |H_{2}| \geq \dots \geq |H_{t}|$,
say with $|H_{1}| = b$. Certainly $b \geq 2$ because otherwise $G
\setminus T^s$ would be a $2$-coloring with
$$
|G \setminus T^s| \geq \ell \cdot 2^{k} + 1 - k\ell = \ell(2^{k} -
k) + 1 \geq 4\ell + 1,
$$
producing the desired monochromatic cycle by
Theorem~\ref{Thm:2-color}. Without loss of generality, suppose red
(color $1$) and blue (color $2$) are the two colors appearing on
edges between parts in the Gallai partition. Note that the choice of
$t$ to be minimum implies that both red and blue are either
connected or absent in the reduced graph so in particular, every
part has red edges to at least one other part and blue edges to at
least one other part (note that if $t = 2$, there would necessarily
be only one such color).

If $2\leq t \leq 3$, then by the minimality of $t$, we may
assume $t=2$, say with corresponding parts $H_{1}$ and
$H_{2}$ with all red (color $1$) edges in between the
two parts. From the maximum $s$, we may assume
$|H_{1}|\geq |H_{2}|\geq \ell$. Note that $T_{1}^s = \emptyset$. By Lemma~\ref{Lemma:SmallOutside} we get \beqs
n & = & |H_{1} \cup T_{k}^s| + |H_{2} \cup T_{2}^s \cup T_{3}^s \cup \dots \cup T_{k - 1}^s|\\
~ & \leq & 2[{\rm gr}_{k - 1}(K_{3} : C_{2\ell + 1}) - 1]\\
~ & = & 2(\ell \cdot 2^{k - 1})\\
~ & < & \ell \cdot 2^{k} + 1, \eeqs a contradiction.

We may therefore assume that $t\geq 4$. Then we have the following claim.

\begin{claim}\label{Claim:LargestPartn}
$|H_{1}| \leq \ell$.
\end{claim}

\begin{proof}
Suppose not, so $|H_{1}| \geq \ell + 1$. Let $r$ be the number of
parts with $H_{i}$ with $|H_{i}| \geq \ell + 1$ so $|H_{1}| \geq
|H_{2}| \geq \dots \geq |H_{r}| \geq \ell + 1$ and call these parts
\emph{large}. Certainly $r \leq 5 = R(K_{3}, K_{3}) - 1$ since any
monochromatic triangle in the reduced graph among these large parts
would yield a monochromatic copy of $C_{2\ell + 1}$. We break the
remainder of the proof into cases based on the value of $r$.

\setcounter{case}{0}
\begin{case}
$r = 5$.
\end{case}
Since there can be no monochromatic triangle in the reduced graph
restricted to the $5$ large parts, the reduced graph must be the
unique $2$-coloring of $K_{5}$ containing two complementary copies
of $C_{5}$. Either one of these cycles yields a monochromatic copy
of $C_{2\ell + 1}$, completing the proof of this case.

\begin{case}\label{Case:r4}
$r = 4$.
\end{case}

If $t = 4$, then there are no vertices in $(G \setminus T^s)
\setminus (H_{1} \cup \dots \cup H_{4})$, so since $|H_{i}| \geq
\ell + 1$, no part can contain any red or blue edges as such an edge
would yield a monochromatic copy of $C_{2\ell + 1}$. This also means
that $T_{1}^s = T_{2}^s = \emptyset$. Define $H_{1}' = H_{1} \cup
T_{3}^s$, $H_{2}' = H_{2} \cup T_{4}^s \cup T_{5}^s \cup \dots \cup
T_{k}^s$ and let $H_{i}' = H_{i}$ for $i \geq 3$. Then $|H_{i}'|
\leq {\rm gr}_{k - 2}(K_{3} : C_{2\ell + 1}) - 1$ for $i \in \{3,
4\}$ and by Lemma~\ref{Lemma:SmallOutside} (which is made possible
by the fact that $|T_{i}^s| \leq \ell$), we have $|H_{i}'| \leq {\rm
gr}_{k - 2}(K_{3} : C_{2\ell + 1}) - 1$ for $i \in \{1, 2\}$. We
therefore get that \beqs
n & = & \sum_{i = 1}^{4} |H_{i}'| \\
~ & \leq & 4[{\rm gr}_{k - 2}(K_{3} : C_{2\ell + 1}) - 1]\\
~ & = & 4(\ell \cdot 2^{k - 2}) ~({\rm by~induction~on}~k)\\
~ & < & \ell \cdot 2^{k} + 1, \eeqs a contradiction. This means that
$t \geq 5$ so there is at least one vertex $v$ in $(G \setminus T^s)
\setminus (H_{1} \cup \dots \cup H_{4})$. The reduced graph
restricted to the large parts could either be two complementary
copies of $P_{4}$ or a matching in one color with all other edges
between parts in the other color.

First suppose the reduced graph consists of a red matching, say
$H_{1}H_{2}$ and $H_{3}H_{4}$ with all blue edges otherwise in
between the parts. In order to avoid creating a red copy of
$C_{2\ell + 1}$ using edges between $H_{1}$ and $H_{2}$ along with
$v$, the vertex $v$ must have all blue edges to one of $H_{1}$ or
$H_{2}$ and similarly to one of $H_{3}$ or $H_{4}$, say $H_{1}$ and
$H_{3}$. Then the blue edges between $H_{1}$ and $H_{3}$ along with
$v$ yield a blue copy of $C_{2\ell + 1}$, for a contradiction.

Finally suppose the reduced graph is two complementary copies of
$P_{4}$, say $H_{1}H_{2}H_{3}H_{4}$ in red and the remaining edges
in blue. To avoid creating a blue copy of $C_{2\ell + 1}$ using the
blue edges between $H_{1}$ and $H_{4}$ along with $v$, the vertex
$v$ must have all red edges to either $H_{1}$ or $H_{4}$, suppose
$H_{1}$. In order to avoid creating a red copy of $C_{2\ell + 1}$
using the red edges between $H_{1}$ and $H_{2}$ along with $v$, the
vertex $v$ must have all blue edges to $H_{2}$. In order to avoid
creating a blue copy of $C_{2\ell + 1}$ using the blue edges between
$H_{2}$ and $H_{4}$ along with $v$, the vertex $v$ must have all red
edges to $H_{4}$. Then $vH_{1}H_{2}H_{3}H_{4}v$ induces a red copy
of $C_{5}$ in the reduced graph, yielding a red copy of $C_{2\ell +
1}$ in $G$ for a contradiction, completing the proof in this case.

\begin{case}
$r = 3$.
\end{case}

The cycle among the $3$ large parts cannot be monochromatic so
suppose, without loss of generality, that the edges from $H_{2}$ to
$H_{3}$ are blue and all other edges between these parts are red.
Let $A$ be the set of vertices in $(G \setminus T^s) \setminus (H_{1} \cup H_{2}
\cup H_{3})$ with blue edges to $H_{1}$ and $H_{3}$ and red edges to
$H_{2}$, let $B$ be the set with red edges to $H_{2}$ and $H_{3}$
and blue edges to $H_{1}$, and let $C$ be the set with blue edges to
$H_{1}$ and $H_{2}$ and red edges to $H_{3}$. Note that any or all
of these sets of vertices may be empty and $G = H_{1} \cup H_{2}
\cup H_{3} \cup A \cup B \cup C \cup T^s$. Also note that $T_{1}^s =
T_{2}^s = \emptyset$.

Either $A$ or $C$ must be empty since the blue edges between $H_{2}$
and $H_{3}$ along with a blue path of the form $H_{2}CH_{1}AH_{3}$
yields a blue copy of $C_{2\ell + 1}$. Without loss of generality,
suppose $C = \emptyset$. Each part of $H_{2}$ and $H_{3}$ contains
no red or blue edges and the induced subgraph $H_{1} \cup A \cup B$
contains no red edges. Let $H_{2}' = H_{2} \cup T_{3}^s$ and let
$H_{3}' = H_{3} \cup T_{4}^s \cup T_{5}^s \cup \dots \cup T_{k}^s$
so by Lemma~\ref{Lemma:SmallOutside} (which is made possible by the
fact that $|T_{i}^s| \leq \ell$), we have $|H_{i}'| \leq {\rm gr}_{k
- 2}(K_{3} : C_{2\ell + 1}) - 1$ for $i \in \{2, 3\}$. This yields
\beqs
n & = & |H_{2}'| + |H_{3}'| + |H_{1} \cup A \cup B|\\
~ & \leq & 2[{\rm gr}_{k - 2}(K_{3} : C_{2\ell + 1}) - 1] + [{\rm gr}_{k - 1}(K_{3} : C_{2\ell + 1}) - 1]\\
~ & = & 2(\ell \cdot 2^{k - 2}) + (\ell \cdot 2^{k - 1})\\
~ & < & \ell \cdot 2^{k} + 1, \eeqs a contradiction, completing the
proof of this case.

\begin{case}
$r = 2$.
\end{case}

Without loss of generality, suppose the edges between $H_{1}$ and
$H_{2}$ are red. Let $A$ be the set of vertices in $(G \setminus T^s)
\setminus (H_{1} \cup H_{2})$ with red edges to $H_{1}$ and blue
edges to $H_{2}$, let $B$ be the set with blue edges to $H_{1} \cup
H_{2}$, and let $C$ be the set with blue edges to $H_{1}$ and red
edges to $H_{2}$. No vertex $u \in B$ can have red edges to both $v
\in A$ and $w \in C$ since the red edges between $H_{1}$ and $H_{2}$
along with a red path of the form $H_{1}vuwH_{2}$ would yield a red
copy of $C_{2\ell + 1}$ in $G$.

Note that $T_{1}^s = \emptyset$. If $|A \cup B| \geq \ell + 1$ (or
similarly $|B \cup C| \geq \ell + 1$), then neither $A \cup B$ nor
$H_{2}$ (respectively neither $B \cup C$ nor $H_{1}$) can contain
any blue edges.

If both $|A \cup B| \geq \ell + 1$ and $|B \cup C| \geq \ell + 1$,
then we claim that one of $A,B,C$ must be an empty set. Assume, to
the contrary, that $A\neq \emptyset$, $B\neq \emptyset$, and $C\neq
\emptyset$. To avoid a blue $C_{2\ell + 1}$ in $A\cup B\cup H_2$,
the edges among the parts in $A\cup B$ are all red. By the same
reason, the edges among the parts in $B\cup C$ are all red. Then
there is a red $C_{2\ell + 1}$, a contradiction. So one of $A,B,C$
must be an empty set.

Suppose that $B = \emptyset$. Assuming the first case, with $|A|
\geq \ell + 1$ and $|C| \geq \ell + 1$, $A$ and $C$ must also
contain no red edges. With no red or blue edges, $A$ and $C$ must
each be single parts of the Gallai partition, contradicting the
assumption that $r = 2$ (since $|A|, |C| \geq \ell + 1$).

Suppose that $B\neq \emptyset$. Without loss of generality, let
$A=\emptyset$. Then $|B|
> \ell$. We claim that $|C|\leq 1$. Assume, to
the contrary, that $|C|\geq 2$. Then the edges among the parts in
$B\cup C$ are all red. Choose $u,v\in B$ such that the edge $uv$ is
red. Choose two vertices in $C$ and they together with $H_1\cup
H_2\cup \{u,v\}$ form a red $C_{2\ell + 1}$, a contradiction. So
$|C|\leq 1$, and hence $|C|=1$, otherwise, contradicting to the
minimality of $t$. So $T_{1}^s = T_{2}^s = \emptyset$ and \beqs
n & = & |H_{1} \cup C | + |H_{2} \cup T_{k}^s| + |B \cup T_{3}^s \cup \dots \cup T_{k - 1}^s|\\
~ & \leq & 3[{\rm gr}_{k - 2}(K_{3} : C_{2\ell + 1}) - 1] \\
~ & = & 3(\ell \cdot 2^{k - 2})\\
~ & < & \ell \cdot 2^{k} + 1, \eeqs a contradiction, meaning that we
may assume $|A| \leq \ell$.

Next suppose one of $|A \cup B|$ or $|B \cup C|$ is at least $\ell +
1$ and the other is not, say $|A \cup B| \geq \ell + 1$. Suppose
further that $|A| \geq \ell + 1$ so $A$ contains no red or blue
edges. By Lemma~\ref{Lemma:SmallOutside} (which is made possible by
the fact that $|T_{i}^s| \leq \ell$), we get \beqs
n & = & |H_{1} \cup B \cup C \cup T_{2}^s| + |H_{2} \cup T_{k}^s| + |A \cup T_{3}^s \cup \dots \cup T_{k - 1}^s|\\
~ & \leq & [{\rm gr}_{k - 1}(K_{3} : C_{2\ell + 1}) - 1] + 2[{\rm gr}_{k - 2}(K_{3} : C_{2\ell + 1}) - 1] \\
~ & = & (\ell \cdot 2^{k - 1}) + 2(\ell \cdot 2^{k - 2})\\
~ & < & \ell \cdot 2^{k} + 1, \eeqs a contradiction. Then again
using Lemma~\ref{Lemma:SmallOutside}, we get \beqs
n & = & |H_{1} \cup B \cup C \cup T_{k}^s| + |H_{2} \cup A \cup T_{2}^s \cup T_{3}^s \cup \dots \cup T_{k - 1}^s|\\
~ & \leq & 2[{\rm gr}_{k - 1}(K_{3} : C_{2\ell + 1}) - 1]\\
~ & = & 2(\ell \cdot 2^{k - 1})\\
~ & < & \ell \cdot 2^{k} + 1, \eeqs a contradiction. In fact, the
same analysis as this last subcase also applies when both $|A \cup
B| \leq \ell$ and $|B \cup C| \leq \ell$. By
Lemma~\ref{Lemma:SmallOutside}, we get \beqs
n & = & |H_{1} \cup B \cup C \cup T_{k}^s| + |H_{2} \cup A \cup T_{2}^s \cup T_{3}^s \cup \dots \cup T_{k - 1}^s|\\
~ & \leq & 2[{\rm gr}_{k - 1}(K_{3} : C_{2\ell + 1}) - 1]\\
~ & = & 2(\ell \cdot 2^{k - 1})\\
~ & < & \ell \cdot 2^{k} + 1, \eeqs completing the proof of this
case.

\begin{case}\label{Case:LastCase}
$r = 1$.
\end{case}

Let $A$ be the set of vertices in $(G \setminus T^s) \setminus
H_{1}$ with red edges to $H_{1}$ and let $B$ be the set with blue
edges to $H_{1}$. If $|A| \leq \ell$ and $|B| \leq \ell$, we move
both $A$ and $B$ to $T^s$, contradicting the maximality of $s$. If
one of these sets is large and the other is not, say $|A| \geq \ell
+ 1$ and $|B| \leq \ell$, then neither $H_{1}$ nor $A$ can contain
any red edge (and $T_{1}^s = \emptyset$). Note that $A$ and $B$
consist of parts of the Gallai partition with order at most $\ell$.
By Lemma~\ref{Lemma:2l+1}, we know that $|A| \leq 2\ell$. So we may
apply Lemma~\ref{Lemma:SmallOutside} (which is made possible by the
fact that $|T_{i}| \leq \ell$ and $|B \cup T_{2}^s| \leq \ell$) to
get $|H_{1} \cup B \cup T_{2}^s \cup T_{3}^s \cup \dots \cup T_{k -
1}^s| \leq {\rm gr}_{k - 1}(K_{3} : C_{2\ell + 1}) - 1$. Putting
these together, we get \beqs
n & = & |A \cup T_{k}^s| + |H_{1} \cup B \cup T_{2}^s \cup T_{3}^s \cup \dots \cup T_{k - 1}^s|\\
~ & \leq & 3\ell+gr_{k - 1}(K_{3} : C_{2\ell + 1}) - 1\\
~ & = & 3\ell+\ell \cdot 2^{k - 1}\\
~ & < & \ell \cdot 2^{k} + 1, \eeqs a contradiction.

Finally suppose $|A| \geq \ell + 1$ and $|B| \geq \ell + 1$ so
$H_{1}$ contains neither red nor blue edges, $A$ contains no red
edge, $B$ contains no blue edge, and $T_{1}^s = T_{2}^s = \emptyset$.
Note that $A$ and $B$ consist of parts of the Gallai partition with
order at most $\ell$. By Lemma~\ref{Lemma:2l+1}, we know that $|A|
\leq 2\ell$ and $|B| \leq 2\ell$. By Lemma~\ref{Lemma:SmallOutside},
we have $|H_{1} \cup T_{3}^s \cup T_{4}^s \cup \dots \cup T_{k - 1}^s|
\leq {\rm gr}_{k - 2}(K_{3} : C_{2\ell + 1}) - 1$ and $|A \cup
T_{k}^s| \leq 3\ell$. Putting these together, we get \beqs
n & = & |H_{1} \cup T_{3}^s \cup T_{4}^s \cup \dots \cup T_{k - 1}^s| + |A| + |T_{k}^s| + |B|\\
~ & \leq & [{\rm gr}_{k - 2}(K_{3} : C_{2\ell + 1}) - 1] + 2\ell + \ell+ 2\ell\\
~ & = & (\ell \cdot 2^{k - 2}) + 5\ell\\
~ & < & \ell \cdot 2^{k} + 1, \eeqs a contradiction. This completes
the proof of Case~\ref{Case:LastCase} and
Claim~\ref{Claim:LargestPartn}.
\end{proof}

By Claim~\ref{Claim:LargestPartn}, we know that all parts in the
Gallai partition of $G\setminus T^s$ have at most $\ell$ vertices,
which means that they are all small parts. Since no such small part
can contain a copy of $C_{2\ell + 1}$, it follows that to get an odd
cycle $C_{2\ell + 1}$, we must use the edges among the small parts.

Suppose $k\geq 4$. Then $n = \ell \cdot 2^{k} + 1$, and
$|G|-|T^s|\geq 12\ell+1$. Let $H_{1}, H_{2}, \ldots, H_{t}$ be parts
of Gallai partition for $G\setminus T^s$. Suppose the edges between
parts of the Gallai partition are colored red or blue. If
$|H_{i}|=1$ for all $i \ (1\leq i\leq t)$, then there is a
monochromatic odd cycle $C_{2\ell + 1}$, a contradiction. Then there
exists some part, say $H_1$, such that $|H_1|\geq 2$. Let $A$ be the
set of vertices in $(G \setminus T^s) \setminus H_{1}$ with red
edges to $H_{1}$ and let $B$ be the set with blue edges to $H_{1}$.
Then $|H_1\cup A|\geq 6\ell$ or $|H_1\cup B|\geq 6\ell$. From Lemma
\ref{Lemma:5l-2}, there is a monochromatic odd cycle $C_{2\ell +
1}$, a contradiction.

We may assume $k = 3$ and so $n = \ell \cdot 2^{3} + 1 = 8\ell + 1$.
Let $H_{1}, H_{2}, \ldots, H_{t}$ be parts of Gallai partition for
$G\setminus T^s$. Suppose the edges between parts of the Gallai
partition are colored red or blue. From Lemmas \ref{Lemma:5l-2} and
\ref{Lemma:7l-3}, $T_1^{s}=\emptyset$ and $T_2^{s}=\emptyset$. For
each part $H_i \ (1\leq i\leq t)$, we let $A_i$ be the set of
vertices with red edges to $H_{i}$ and let $B_i$ be the set of
vertices with blue edges to $H_{i}$. Note that $|T_{3}^s|=|T^s|$.

\begin{claim}\label{Claim:Tl2}
$|T_{3}^s| \leq \frac{\ell}{2}$.
\end{claim}

\begin{proof}
Suppose, for a contradiction, that $|T_{3}^s| > \frac{\ell}{2}$ and
call color $3$ green. Let $m = |T_{3}^s|$ so by
Claim~\ref{Claim:Tissmall}, this means that $\frac{\ell}{2} < m \leq
\ell$. Let $x$ be the maximum number of disjoint copies of $P_{3}$
in green in $G \setminus T^s$. We claim that $x\leq 2\ell - 3m$.
Assume, to the contrary, that $x\geq 2\ell - 3m+1$. Choose $2\ell -
3m+1$ disjoint copies of $P_{3}$ in green in $G \setminus T^s$, say
$P_3^{i}=u_iv_iw_i \ (1\leq i\leq 2\ell - 3m+1)$. Let $y$ be the
maximum number of disjoint copies of $P_{2}$  in green in $G
\setminus T^s \setminus \{P_3^{i}\,|\,1\leq i\leq 2\ell - 3m+1\}$,
say $P_2^{i}=a_ib_i \ (1\leq i\leq y)$. We claim that $y\geq
4m-2\ell-1$. Assume, to the contrary, that $y\leq 4m-2\ell-2$. By
deleting all vertices $\{P_3^{i}\,|\,1\leq i\leq x\}\cup
\{P_2^{i}\,|\,1\leq i\leq y\}$ in $G \setminus T^s$, there is no
green edges in the resulting graph, and hence \beqs |G \setminus T^s
\setminus\{P_3^{i}\,|\,1\leq i\leq x\}\setminus
\{P_2^{i}\,|\,1\leq i\leq y\}| & \geq & (8\ell + 1) - \ell-3(2\ell - 3m+1) \\
~ & ~ &-2(4m - 2\ell - 2)\\
~ & = & 5\ell + m + 2\\
~ & > & 4\ell + 1. \eeqs By Theorem~\ref{Thm:2-color}, there is a
monochromatic copy of $C_{2\ell + 1}$ in the reduced graph,  a
contradiction. Let $T_{3}^s=\{d_1,d_2,\ldots,d_m\}$. Then
$d_1a_1b_1d_2a_2b_2d_3\ldots
a_{4m-2\ell-1}b_{4m-2\ell-1}$\\
$d_{4m-2\ell}u_1v_1w_1\ldots u_{2\ell - 3m+1}v_{2\ell -
3m+1}w_{2\ell - 3m+1}d_1$ is a green odd cycle $C_{2\ell + 1}$, a
contradiction. This means that we can delete a set $F$ of at most
$6\ell - 9m$ vertices from $G \setminus T^s$ to leave behind at most
a matching in green, and hence there is a Gallai partition of $G
\setminus (T^s \cup F)$ in which all parts have order at most $2$.
We consider such a partition for the remainder of this proof of Claim~\ref{Claim:Tl2}.

It is clear that $|G \setminus (T^s \cup F)| \geq 6\ell + 2$. Let
$r$ be the number of parts of order $2$ in the Gallai partition of
$G \setminus (T^s \cup F)$. We claim that $r \leq 2\ell + 2$.
Assume, to the contrary, that $r\geq 2\ell + 3$. From
Theorem~\ref{Thm:2-color}, there would exist a monochromatic odd
cycle of length $\ell + 1$ or $\ell + 2$ in the reduced graph.
Without loss of generality, we suppose that $\ell + 1$ is odd. Note
that each part has two vertices and there is an odd cycle
$C_{2\ell+1}$, a contradiction.

Since there are at least $6\ell + 2$ vertices, this means there are
actually at least
$$
(6\ell + 2) - 2(2\ell + 2) = 2\ell - 2
$$
parts of order $1$. Since $\ell \geq 6$, there is at least one such
part, say $X$. If $r = 2\ell + 2$, then the reduced graph on these
parts along with $x$ ($X=\{x\}$) must produce a monochromatic odd
cycle of length $\ell + 1$ or $\ell + 2$, making a monochromatic
copy of $C_{2\ell + 1}$ in $G$. This means $r \leq 2\ell + 1$ so
there are at least
$$
(6\ell + 2) - 2(2\ell + 1) = 2\ell
$$
parts of order $1$. Therefore, there must be at least $(2\ell + 1) +
2\ell = 4\ell + 1$ parts in the Gallai partition. By
Theorem~\ref{Thm:2-color}, there is a monochromatic copy of
$C_{2\ell + 1}$ in the reduced graph, and therefore in $G$.
\end{proof}

Let $H_{1}=\{d_1,d_2,\ldots,d_b\}$. From Lemma~\ref{Lemma:5l-2}, we
have $|H_i|\leq \ell$ for all $i \ (1\leq i\leq t)$, and hence
$|H_1|+|G_R|\leq 5\ell-2$ and $|G_{B}|=|G|-|T^s|-|H_1|\geq 2\ell$.

Our final claim shows that the largest part $H_{1}$ is even smaller
than previously claimed.

\begin{claim}\label{Claim:LargestPartn/2}
$|H_{1}| \leq \frac{\ell}{2}$.
\end{claim}

\begin{proof}
Let $G_{R}$ (and $G_{B}$) be the sets of parts in $G \setminus
(H_{1} \cup T_{3}^s)$ with red (or blue respectively) edges to
$H_{1}$, say with $|G_{R}| \geq |G_{B}|$. We first show that $b =
|H_{1}| \leq \frac{2\ell}{3}$ so suppose that $\frac{2\ell + 1}{3}
\leq b \leq \ell$. By Claim~\ref{Claim:Tl2}, we know that $|T_{3}^s|
\leq \frac{\ell}{2}$ so
$$
G_{R} \geq \frac{(8\ell + 1) - \frac{\ell}{2} - b}{2}.
$$
Suppose that there are at least $2\ell + 1 - 2b$ disjoint red edges
within $G_{R}$, say $P_2^{i}=a_ib_i \ (1\leq i\leq 2\ell + 1 - 2b)$.
Then $d_1a_1b_1d_2a_2b_2d_3\ldots a_{2\ell + 1 - 2b}b_{2\ell + 1 -
2b}d_{2\ell + 2 - 2b}u_1\ldots $\\ $u_{3b - 2\ell - 1}d_1$ is a red
odd cycle $C_{2\ell + 1}$, where $u_1,u_2,\ldots,u_{3b - 2\ell - 1}$
are $3b - 2\ell - 1$ vertices in $G_R\setminus (\cup_{i=1}^{2\ell +
1 - 2b} P_2^{i})$, a contradiction. See
Figure~\ref{Fig:Construction} for an example of this construction.
We may therefore delete at most $2(2\ell - 2b)$ vertices from
$G_{R}$ to produce a subgraph $G_{R}'$ containing no red edges with
\beqs
|G_{R}'| & \geq & \frac{8\ell + 1 - \frac{\ell}{2} - b}{2} - 4\ell + 4b\\
~ & = & \frac{14b - \ell + 2}{4}\\
~ & \geq & \frac{25\ell + 20}{12} > 2\ell + 1 \eeqs so by
Lemma~\ref{Lemma:2l+1}, there is a blue copy of $C_{2\ell + 1}$
within $G_{R}'$, a contradiction. Thus, we may assume that
$\frac{\ell + 1}{2} \leq |H_{1}| \leq \frac{2\ell}{3}$.

\begin{figure}[H]
\begin{center}
\includegraphics{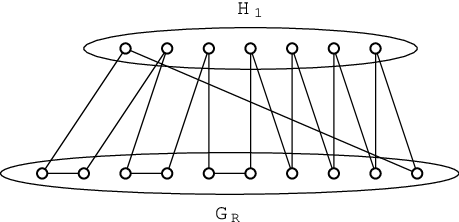}
\caption{Construction of a red copy of $C_{2\ell + 1}$
\label{Fig:Construction}}
\end{center}
\end{figure}

Let $P^{1}$ be a longest red path within $G_{R}$ and let $P^{2}$ be
a longest blue path within $G_{B}$. Note that $|P^{i}| \leq 2(\ell -
b) + 1 \leq \ell$ for $i \in \{1, 2\}$ to avoid creating a
monochromatic copy of $C_{2\ell + 1}$ and $|P^{1}|, |P^{2}| \geq 2$
by Lemma~\ref{Lemma:2l+1}. Let $H_{R_{1}}$ and $H_{R_{2}}$ be the
parts of the Gallai partition of $G \setminus T^s$ which contain the
end vertices of $P^{1}$ (where it is possible that $H_{R_{1}} =
H_{R_{2}}$) and similarly let $H_{B_{1}}$ and $H_{B_{2}}$ be the
parts containing the end vertices of $P^{2}$. Let $u_{1}u_{2} \dots
u_{|P^{1}|}$ be the vertices of $P^{1}$ and let $v_{1}v_{2} \dots
v_{|P^{2}|}$ be the vertices of $P^{2}$. Let $w_{1} \in T_{3}^s$ (if
$T_{3}^s \neq \emptyset$), $w_{2}, w_{3} \in H_{1}$, $w_{4} \in
G_{R} \setminus P^{1}$, $w_{5} \in G_{B} \setminus P^{2}$, and
$w_{6}w_{7}$ be any green edge within $G \setminus T_{3}^s$ (note
that such an edge must exist since otherwise $G\setminus T_{3}^s$ is
$2$-colored with more than $4\ell + 1$ vertices).

Let $G' = G \setminus (P^{1} \cup P^{2} \cup \{w_{1}, w_{2}, \dots,
w_{7}\})$ so $|G'| \geq (8\ell + 1) - (2\ell + 7) > 6\ell - 9.$ By
Theorem~\ref{Thm:EvenUp}, there is a monochromatic copy of $C_{2\ell
- 2}$ in $G'$. Note that if there is a green copy of $C_{2\ell -
2}$, then $C_{2\ell - 2}$ must use edges from $T_{3}^s$ to $G'
\setminus T_{3}^s$, since each part of the Gallai partition of $G'
\setminus T^s$ has at most $\frac{2\ell}{3}$ vertices. Let $e = uv$
be one such edge where $u \in T_{3}^s$ and $v \notin T_{3}^s$. By
replacing the edge $uv$ by the path $uw_6w_7w_1v$, we can get a
green copy of $C_{2\ell + 1}$, a contradiction. Then $C_{2\ell - 2}$
is red or blue. Without loss of generality, we assume that $C_{2\ell
- 2}$ is red. If there exists a vertex $u$ in $H_{1}$ such that $u
\in H_{1}\cap C_{2\ell - 2}$, then it follows from $C_{2\ell -
2}\nsubseteq H_{1}$ that there is an red edge of $C_{2\ell - 2}$
from $H_{1}$ to $G \setminus H_{1}$, say $uv$, where $v \in G_{R}$.
By replacing the edge $uv$ by the red path $uu_{1}u_{2}w_{3}v$, we
can get a red copy of $C_{2\ell + 1}$, a contradiction. If there are
two vertices in $G_{R}\cap C_{2\ell - 2}$, say $u$ and $v$, such
that their distance in the cycle $C_{2\ell - 2}$ is $d \ (d \leq
\ell-1)$, then by replacing this path by the red path
$ud_{1}f_{1}d_{2}f_{2}\dots f_{\frac{d}{2}}d_{\frac{d}{2}+1}v$ (if
$d$ is even) or $ud_{1}f_{1}d_{2}f_{2}\dots
f_{\frac{d-3}{2}}d_{\frac{d-3}{2}+1}u_{1}u_{2}d_{\frac{d-3}{2}+2}v$
(if $d$ is odd), we can get a red copy of $C_{2\ell + 1}$, where
$f_1,f_2,\ldots,f_{\frac{d}{2}}$ are $\frac{d}{2}$ vertices in
$G_{R}$ or $f_1,f_2,\ldots,f_{\frac{d-3}{2}}$ are $\frac{d-3}{2}$
vertices in $G_{R}\setminus \{u_{1},u_{2}\})$, a contradiction. Then
$|C_{2\ell - 2} \cap G_{R}|\leq 1$, and hence $2\ell - 3\leq
|C_{2\ell - 2} \cap G_{B}|\leq 2\ell - 2$. If $B_{1}=B_{2}$, then
$C_{2\ell - 2}$ contains a red edge $uv \in G_{B} \setminus (P^{2}
\cup H_{B_{1}} \cup H_{B_{2}})$, and so by replacing the edge with a
red path $uv_{1}w_{5}v_{|P^{2}|}v$, we can get a red copy of
$C_{2\ell + 1}$, a contradiction. If $B_{1}\neq B_{2}$, then we can
find an red edge $uv$ such that $uv \in G_{B} \setminus (P^{2} \cup
H_{B_{1}} \cup H_{B_{2}})$, or $u\in H_{B_{1}}\cup H_{B_{2}}$ and
$v\in G_{B} \setminus (P^{2} \cup H_{B_{1}} \cup H_{B_{2}})$. By
replacing the edge with a red path $uv_{1}w_{5}v_{|P^{2}|}v$ or
$uw_{5}v_{1}v_{|P^{2}|}v$, we can get a red copy of $C_{2\ell + 1}$,
a contradiction. Regardless of the color (red, blue, or green) and
location with respect to the sets (i.e. $H_{1}, G_{R}, G_{B},
\dots$), we may construct a monochromatic copy of $C_{2\ell + 1}$
using this monochromatic copy of $C_{2\ell - 2}$ and some vertices
of $G \setminus G'$, a contradiction.
\end{proof}

Let $G_{R}$ and $G_{B}$ denote the sets of vertices in $G \setminus
H_{1}$ with red and blue edges respectively to $H_{1}$. Without loss
of generality, suppose $|G_{R}| \geq |G_{B}|$.

From Claim \ref{Claim:LargestPartn/2}, $|H_{1}| \leq \frac{\ell}{2}$
for all $i \ (1\leq i\leq t)$. From Lemma~\ref{Lemma:5l-2}, we have
$|H_1|+|G_R|\leq 4\ell+\frac{\ell}{2}-2$, and hence $|G_{B}| \geq
\frac{5\ell}{2} + 6$. Let $P^{1}$ be a longest red path within
$G_{R}$ and let $P^{2}$ be a longest blue path within $G_{B}$. Note
that $|P^{i}| \leq 2(\ell - b) + 1$ for $i \in \{1, 2\}$ to avoid
creating a monochromatic copy of $C_{2\ell + 1}$ and $|P^{1}|,
|P^{2}| \geq 2$ by Lemma~\ref{Lemma:2l+1}. Let $H_{R_{1}}$ and
$H_{R_{2}}$ be the parts of the Gallai partition of $G \setminus
T^s$ which contain the end vertices of $P^{1}$ (where it is possible
that $H_{R_{1}} = H_{R_{2}}$) and similarly let $H_{B_{1}}$ and
$H_{B_{2}}$ be the parts containing the end vertices of $P^{2}$. Let
$u_{1}u_{2} \dots u_{|P^{1}|}$ be the vertices of $P^{1}$ and let
$v_{1}v_{2} \dots v_{|P^{2}|}$ be the vertices of $P^{2}$.

Let $G_{B}' = G_{B} \setminus (P^{2} \cup H_{B_{1}} \cup H_{B_{2}})$
so $|G_{B}'| \geq (\frac{5\ell}{2} + 6) - |P^{2}| - 2b + 2 = [2(\ell
- b) + 1 - |P^{2}|] + [\frac{\ell}{2} + 10] - 3$. By
Lemma~\ref{Lemma:ColE-G}, there is either a blue path of order
$2(\ell - b) + 1 - |P^{2}|$ or a red path of order at least
$\frac{\ell}{2} + 10$. The blue path, in combination with $P^{2}$
and $H_{1}$, would produce a blue copy of $C_{2\ell + 1}$ so we may
assume there is a red path in $G_{B}$ of order at least
$\frac{\ell}{2} + 10$. By the same argument, there is a blue path of
order at least $\frac{\ell}{2} + 10$ in $G_{R}$ since $|G_{R}| \geq
|G_{B}|$. In particular, these observations mean that there is a red
edge $x_{1}y_{1}$ in $G_{B} \setminus (P^{2} \cup H_{B_{1}} \cup
H_{B_{2}})$ and a blue edge $x_{2}y_{2}$ in $G_{R} \setminus (P^{1}
\cup H_{R_{1}} \cup H_{R_{2}})$.

Let $w_{1} \in T_{3}^s$ (if $T_{3}^s \neq \emptyset$), $w_{2}w_{3}$ be any green edge within $G \setminus T_{3}^s$ (note that such an edge must exist since otherwise $G\setminus T_{3}^s$ is $2$-colored with more than $4\ell + 1$ vertices), $w_{4}, w_{5} \in H_{1}$, and let $x_{1}y_{1}$ and $x_{2}y_{2}$ be defined as above. 
Let $Q^{1} = \{u_{1}, u_{2}, \dots, u_{\ell - 1}, u_{|P^{1}|}\}$ and
let $Q^{2} = \{v_{1}, v_{2}, \dots, v_{\ell - 1}, v_{|P^{2}|}\}$
where some vertices in these sets may not exist if $|P^{i}| < \ell$.
Finally let
$$
G' = G \setminus (\{w_{1}, w_{2}, w_{3}, w_{4}, w_{5}, x_{1}, y_{1},
x_{2}, y_{2}\} \cup Q^{1} \cup Q^{2})
$$
so $|G'| \geq |G| - (9 + 2\ell) > 6\ell - 9$. By
Theorem~\ref{Thm:EvenUp}, there exists a monochromatic copy of
$C_{2\ell - 2}$ within $G'$, say $C$.


If $C$ is green, then since all parts of the Gallai partition of $G
\setminus T^s$ have order at most $\frac{\ell}{2}$, $C$ must use
edges from $T_{3}^s$ to $G' \setminus T_{3}^s$. Let $e = uv$ be one
such edge with $u \in T_{3}^s$ and $v \notin T_{3}^s$. Then
replacing the edge $uv$ with the path $uw_{2}w_{3}w_{1}v$ produces a
green copy of $C_{2\ell + 1}$.

Since we are no longer applying the assumption that $|G_{R}| \geq
|G_{B}|$, we may assume, without loss of generality, that $C$ is red
and claim that a symmetric argument would hold if $C$ was blue. If
$C$ contains two vertices in $G_{R}$, say $u$ and $v$, at distance
(along $C$) at most $2$, then replacing this path with the red path
$uw_{4}u_{1}w_{5}v$ or $uw_{4}u_{1}u_{2}w_{5}v$ (of the appropriate
length) produces a red copy of $C_{2\ell + 1}$. Thus, there can be
no two vertices in $C \cap G_{R}$ at distance at most $2$ along $C$.
More generally, if $P^{1}$ is longer then there can be no two
vertices $u$ and $v$ in $G_{R} \cap C$ at distance at most $|P^{1}|$
along $C$.

Similarly, if there is an edge of $C$ from $H_{1}$ to $G_{R}$, say
$uv$ with $u \in H_{1}$ and $v\in G_{R}$, then we replace this edge
with the red path $uu_{1}u_{2}w_{4}v$ to produce a red copy of
$C_{2\ell + 1}$. This means there is no edge from $H_{1}$ to $G_{R}$
on $C$ and therefore, $C \cap H_{1} = \emptyset$.

If $C$ contains two vertices $u$ and $v$ with $u, v \in G_{B}
\setminus (P^{2} \cup H_{B_{1}} \cup H_{B_{2}})$ at distance at most
$2$ on $C$, then replacing the edge with a red path
$uv_{1}x_{1}v_{|P^{2}|}v$ or $uv_{1}x_{1}y_{1}v_{|P^{2}|}v$ of the
appropriate length produces a red copy of $C_{2\ell + 1}$. This
means $C$ cannot have two vertices at distance at most $2$ in $G_{B}
\setminus (P^{2} \cup H_{B_{1}} \cup H_{B_{2}})$. If $C$ contains
any vertex $u \in H_{B_{i}}$ for $i \in \{1, 2\}$, then since $C$
cannot be contained entirely within $H_{B_{i}}$, there is an edge of
$C$ from $H_{B_{i}}$ to $G \setminus H_{B_{i}}$, say $uv$ where $v
\notin H_{B_{i}}$. Supposing $i = 1$ without loss of generality, we
can replace the edge $uv$ with the red path $ux_{1}y_{1}v_{1}v$ to
produce a red copy of $C_{2\ell + 1}$. We therefore know that $C
\cap (H_{B_{1}} \cup H_{B_{2}}) = \emptyset$.

If $|P^{1}| \geq 5$ or $|H_{1}| \geq 4$, then we claim that the
distance of any two vertices in $C\cap G_{R}$ is at least $6$ in
$C$. Suppose, for a contradiction, that there exists two vertices
$u,v\in C\cap G_{R}$ such that their distance in $C$ is at most $5$.
Let $d$ be the distance between $u$ and $v$ in $C$. By replacing
this path by the red path $uw_{4}u_{1}u_{2}\dots u_{d}w_{5}v$ (if
$|P^{1}| \geq 5$) or $ud_{1}f_{1}d_{2}f_{3}\dots
f_{\frac{d-2}{2}}d_{\frac{d-2}{2}+1}u_{1}u_{2}d_{\frac{d-2}{2}+2}v$
(if $|H_{1}| \geq 4$ and $d$ is even) or $ud_{1}f_{1}d_{2}f_{3}\dots
f_{\frac{d-1}{2}+1}d_{\frac{d-1}{2}+2}v$ (if $|H_{1}| \geq 4$ and
$d$ is odd), we can get a red copy of $C_{2\ell + 1}$, a
contradiction, where $f_1,f_2,\ldots,f_{\frac{d-2}{2}}$ are
$\frac{d-2}{2}$ vertices in $G_{R}$ or
$f_1,f_2,\ldots,f_{\frac{d-1}{2}}$ are $\frac{d-1}{2}$ vertices in
$G_{R}\setminus \{u_{1},u_{2}\})$ and $d_1,d_2,\ldots,d_{b}$ are $b$
vertices in $H_{1}$. Thus, the distance of any two vertices in
$C\cap G_{R}$ is at least $6$ in $C$, and the distance of any two
vertices in $C\cap (G_{B} \setminus (P^{2} \cup H_{B_{1}} \cup
H_{B_{2}}))$ is at least $3$ in $C$. Then there are at most
$\frac{|C|}{6} + \frac{|C|}{3}$ vertices of $C$ belonging to $G'
\setminus P^{2}$. This means that there are at least $\frac{|C|}{2}$
vertices of $C$ belonging to $P^{2}$. Note that $|P^{2} \cap G'|
\leq 2(\ell - b) + 1 - \ell < \ell-1$ and $|C| = 2\ell - 2$, which
is impossible. Thus, we may assume $|P^{1}| \leq 4$ and $|H_{1}|
\leq 3$. Then $|P^{2}| \leq 2(\ell - b) + 1 \leq 2\ell - 3$.

At this point, we may redefine $Q^{2}$ to be $Q^{2 *} = \{v_{1},
v_{2}, \dots, v_{2\ell - 5}, v_{|P^{2}|}\}$ and redefine $G'$ to be
$$
G^{\prime *} = G \setminus (\{w_{1}, w_{2}, w_{3}, w_{4}, w_{5},
x_{1}, y_{1}, x_{2}, y_{2}\} \cup P^{1} \cup Q^{2 *})
$$
Then the distance of any two vertices in $C\cap G_{R}$ is at least
$3$ in $C$, and the distance of any two vertices in $C\cap (G_{B}
\setminus (P^{2} \cup H_{B_{1}} \cup H_{B_{2}}))$ is at least $3$ in
$C$. Furthermore, there are at most $\frac{|C|}{3} + \frac{|C|}{3}$
vertices of $C$ belonging to $G' \setminus P^{2}$. This means that
there are at least $\frac{|C|}{3}$ vertices of $C$ belonging to
$P^{2}$. Note that $|P^{2} \cap G'| \leq 2(\ell - b) + 1 - (2\ell-4)
\leq 1$ and $|C| = 2\ell - 2$, which is impossible. This completes
the proof of Theorem~\ref{Thm:OddCycles}.
\end{proof}

\subsection*{Acknowledgements}

The authors are deeply thankful to the referees for their careful reading and thoughtful suggestions, leading to great improvements in the presentation of this work.

Data sharing not applicable to this article as no datasets were generated or analysed during the current study.


\begin{thebibliography}{10}

\bibitem{AI08}
M.~Axenovich and P.~Iverson.
\newblock Edge-colorings avoiding rainbow and monochromatic subgraphs.
\newblock {\em Discrete Math.}, 308(20):4710--4723, 2008.

\bibitem{Song2}
C.~Bosse and Z.~Song.
\newblock {M}ulticolor {G}allai-{R}amsey numbers of ${C}_9$ and ${C}_{11}$.
\newblock {\em Submitted}.

\bibitem{Song1}
D.~Bruce and Z.~Song.
\newblock {G}allai-{R}amsey numbers of ${C}_7$ with multiple colors.
\newblock {\em Submitted}.

\bibitem{MR1464337}
K.~Cameron and J.~Edmonds.
\newblock Lambda composition.
\newblock {\em J. Graph Theory}, 26(1):9--16, 1997.

\bibitem{CLZ11}
G.~Chartrand, L.~Lesniak, and P.~Zhang.
\newblock {\em Graphs \& digraphs}.
\newblock CRC Press, Boca Raton, FL, fifth edition, 2011.

\bibitem{MR729784}
F.~R.~K. Chung and R.~L. Graham.
\newblock Edge-colored complete graphs with precisely colored subgraphs.
\newblock {\em Combinatorica}, 3(3-4):315--324, 1983.

\bibitem{MR0047308}
G.~A. Dirac.
\newblock Some theorems on abstract graphs.
\newblock {\em Proc. London Math. Soc. (3)}, 2:69--81, 1952.

\bibitem{MR0114772}
P.~Erd{\H{o}}s and T.~Gallai.
\newblock On maximal paths and circuits of graphs.
\newblock {\em Acta Math. Acad. Sci. Hungar}, 10:337--356, 1959.

\bibitem{MR0345866}
R.~J. Faudree and R.~H. Schelp.
\newblock All {R}amsey numbers for cycles in graphs.
\newblock {\em Discrete Math.}, 8:313--329, 1974.

\bibitem{FM11}
S.~Fujita and C.~Magnant.
\newblock Gallai-{R}amsey numbers for cycles.
\newblock {\em Discrete Math.}, 311(13):1247--1254, 2011.

\bibitem{MR2606615}
S.~Fujita, C.~Magnant, and K.~Ozeki.
\newblock Rainbow generalizations of {R}amsey theory: a survey.
\newblock {\em Graphs Combin.}, 26(1):1--30, 2010.

\bibitem{FMO14}
S.~Fujita, C.~Magnant, and K.~Ozeki.
\newblock Rainbow generalizations of {R}amsey theory - a dynamic survey.
\newblock {\em Theo. Appl. Graphs}, 0(1), 2014.

\bibitem{MR0221974}
T.~Gallai.
\newblock Transitiv orientierbare {G}raphen.
\newblock {\em Acta Math. Acad. Sci. Hungar}, 18:25--66, 1967.

\bibitem{GSSS10}
A.~Gy{\'a}rf{\'a}s, G.~S{\'a}rk{\"o}zy, A.~Seb{\H o}, and S.~Selkow.
\newblock Ramsey-type results for gallai colorings.
\newblock {\em J. Graph Theory}, 64(3):233--243, 2010.

\bibitem{MR2063371}
A.~Gy{\'a}rf{\'a}s and G.~Simonyi.
\newblock Edge colorings of complete graphs without tricolored triangles.
\newblock {\em J. Graph Theory}, 46(3):211--216, 2004.

\bibitem{MR3121112}
M.~Hall, C.~Magnant, K.~Ozeki, and M.~Tsugaki.
\newblock Improved upper bounds for {G}allai-{R}amsey numbers of paths and
  cycles.
\newblock {\em J. Graph Theory}, 75(1):59--74, 2014.

\bibitem{MR1846919}
G.~K\'arolyi and V.~Rosta.
\newblock Generalized and geometric {R}amsey numbers for cycles.
\newblock {\em Theoret. Comput. Sci.}, 263(1-2):87--98, 2001.
\newblock Combinatorics and computer science (Palaiseau, 1997).

\bibitem{MR0332567}
V.~Rosta.
\newblock On a {R}amsey-type problem of {J}. {A}. {B}ondy and {P}. {E}rd{\H
  o}s. {I}, {II}.
\newblock {\em J. Combinatorial Theory Ser. B}, 15:94--104; 15 (1973),
  105--120, 1973.

\end{thebibliography}

\end{document}